\documentclass{amsart}
\usepackage[utf8]{inputenc}

\usepackage{graphics}
\usepackage{thmtools}
\usepackage{wasysym}
\usepackage[T1]{fontenc}    % use 8-bit T1 fonts

\usepackage{amsthm}
\usepackage{amsbsy,amsmath,amssymb,amscd,amsfonts}
\usepackage[pagebackref=true]{hyperref}
\usepackage[nameinlink,capitalize,noabbrev]{cleveref}

\usepackage{graphicx,float,latexsym,color}
\usepackage[font={scriptsize,it}]{caption}
\usepackage{subcaption}

\usepackage{makecell}

\usepackage[dvipsnames]{xcolor}

\newtheorem{theorem}{Theorem}
\newtheorem*{theorem*}{Theorem}
\newtheorem{observation}{Observation}
\newtheorem{proposition}{Proposition}
\newtheorem{conjecture}{Conjecture}
\newtheorem{corollary}{Corollary}
\newtheorem{lemma}{Lemma}
\theoremstyle{remark}

\theoremstyle{definition}
\newtheorem{definition}{Definition}

\hypersetup{
    pdftoolbar=true,        % show Acrobat’s toolbar?
    pdfmenubar=true,        % show Acrobat’s menu?
    pdffitwindow=false,     % window fit to page when opened
    pdfstartview={FitH},    % fits the width of 
    colorlinks=true,       % false: boxed links; true: colored links
    linkcolor=OliveGreen,          % color of internal links (change box color with linkbordercolor)
    citecolor=blue,        % color of links to bibliography
    filecolor=black,      % color of file links
    urlcolor=red           % color of external links
}

\usepackage{lineno}

\usepackage{comment}
\usepackage{microtype}
\usepackage{footnote}

\newcommand{\E}{\mathcal{E}}
\renewcommand{\H}{\mathcal{H}}

\newcommand{\C}{\mathcal{C}}
\renewcommand{\O}{\mathcal{O}}

\renewcommand{\P}{\mathcal{P}}

\title{New Invariants of Poncelet-Jacobi\\Bicentric Polygons}
\author[P. Roitman]{Pedro Roitman}
\author[R. Garcia]{Ronaldo Garcia}
\author[D. Reznik]{Dan Reznik} 

\date{March, 2021}

\begin{document}

\maketitle

\begin{abstract}
The 1d family of Poncelet polygons interscribed between two circles is known as the Bicentric family. Using elliptic functions and Liouville's theorem, we show (i) that this family has invariant sum of internal angle cosines and (ii) that the pedal polygons with respect to the family's limiting points have invariant perimeter. Interestingly, both (i) and (ii) are also properties of elliptic billiard N-periodics. Furthermore, since the pedal polygons in (ii) are identical to inversions of elliptic billiard N-periodics with respect to a focus-centered circle, an important corollary is that (iii) elliptic billiard focus-inversive N-gons have constant perimeter. Interestingly, these also conserve their sum of cosines (except for the N=4 case).

\end{abstract}

\section{Introduction}
\label{sec:intro}
The bicentric family is a 1d family of Poncelet N-gons interscribed between two specially-chosen circles \cite[Poncelet's Porism]{mw}. The special case of a family of triangles with fixed incircle and circumcircle was originally studied by Chapple 80 years before Poncelet \cite{odehnal2011-poristic}. Any pair of conics with at least two complex conjugate points of intersection can be sent to a pair of circles via a suitable projective transformation \cite{akopyan2007-conics}. Based on this, in the 1820s Jacobi produced an alternative proof to Poncelet's Great theorem based on simplifications afforded by his elliptic functions over the bicentric family \cite{bos-1987,dragovic11,nash2018-poncelet}.

Referring to
\cref{fig:confocal}, a known fact is that the {\em polar image}\footnote{The polar of a point $P$ with respect to a circle $\C$ centered on $O$ is the line $L$ containing the inversion of $P$ wrt $\C$ and perpendicular to $OP$.} of two non-intersecting circles with respect to either one of their {\em limiting points} is a pair of confocal conics with a focus coinciding with the limiting point chosen \cite{akopyan2007-conics} (see \cref{app:polar-pedal}).

Recall that a pair of non-intersecting circles $\C_1$ and $\C_2$ is associated with a pair of {\em limiting points} $\ell_1,\ell_2$ which, if taken as centers of inversion\footnote{Note that $\ell_1,\ell_2$ coincide with the two points of intersection of all circles orthogonal to $\C_1$ and $\C_2$. This implies that the abovementioned inversions will result in concentric circles.}, send the original circles to two distinct pairs of concentric circles \cite[Limiting Points]{mw}.

Conversely, the bicentric family is the polar image of elliptic (or hyperbolic) billiard N-periodics with respect to a circle centered on a focus (see \cref{sec:five-polys}). Recall the latter conserve both perimeter\footnote{Billiard inscribed in hyperbolas conserve {\em signed} perimeter, see \cref{sec:five-polys}.} and Joachimsthal's constant \cite{sergei91}.

\begin{figure}
    \centering
    \includegraphics[width=.9\textwidth]{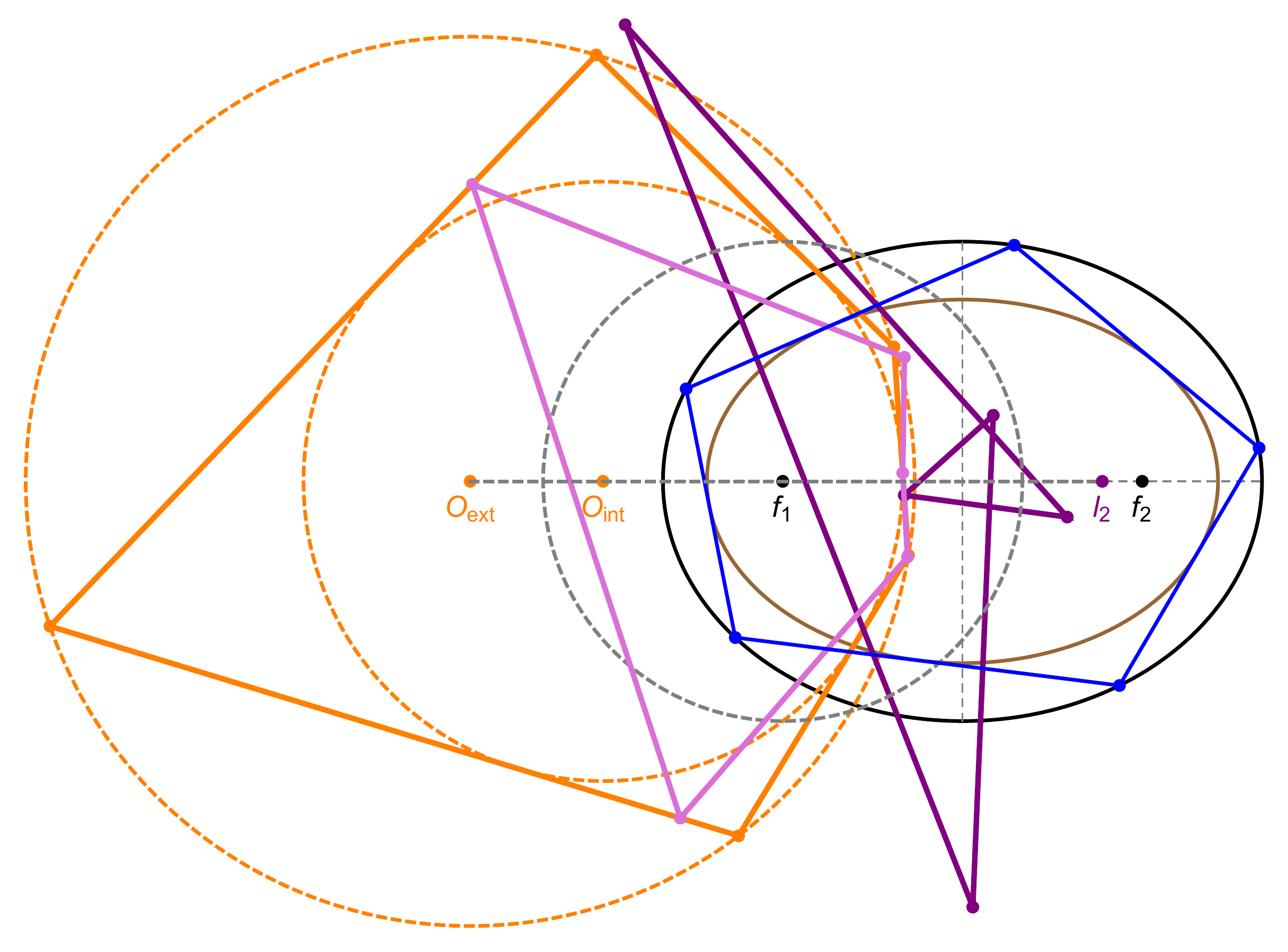}
    \caption{The bicentric family  (solid orange) is the polar image of elliptic billiard N-periodics (blue) with respect to a circle (dashed gray) centered on $f_1$ (which coincides with limiting point $\ell_1$). Also shown are the constant-perimeter bicentric pedals (pink and purple) with respect to either limiting point, $f_1=\ell_1$ and $\ell_2$.  \href{https://youtu.be/8m21fCz8eX4}{Video}}
    \label{fig:confocal}
\end{figure}

\textbf{Main Results:}
Though the bicentric family was much studied in the last 200 years, interactive experimentation with their dynamic geometry has led us to detect and prove a few new curious facts, perhaps known to the giants of the XIX century but never jotted down.

\begin{itemize}
    \item \cref{thm:bicentric-sum}: The sum of the cosines of bicentric polygons is invariant over the family. This mirrors an invariant recently proved for elliptic billiard N-periodics \cite{reznik2020-intelligencer,garcia2020-new-properties,akopyan2020-invariants,bialy2020-invariants}.
    \item  \cref{thm:bicentric-pedal-perimeter} The perimeter of pedal polygons of the bicentrics with respect to its limiting points is invariant; see \cref{fig:confocal}. Notice this too mirrors perimeter invariance of elliptic billiard N-periodics.
    \item \cref{cor:inv-per}: bicentric pedals with respect to a limiting point are identical to the inversion of billiard N-periodics with respect to a focus, therefore the latter also conserves perimeter. In fact it was this surprising observation (see this \href{https://youtu.be/wkstGKq5jOo}{Video}) that prompted the current article.
    \item \cref{conj:limiting-sum-cosines}: Experiments show that the two limiting pedal polygons also conserve their sum of cosines, except for  the case of the $N=4$ pedal with respect to $\ell_1$.
\end{itemize}

\subsection*{Article Structure}
In \cref{sec:jacobi}, we review Jacobi's parametrization for bicentric polygons. We then use it to obtain expressions in terms of Jacobi elliptic functions for each of the above invariants, see \cref{sec:bicentric-sum-of-cosines,sec:pedal-perimeter}. \cref{sec:five-polys} paints a unified view of the five polygon families mentioned herein. A list of illustrative videos appear in \cref{sec:videos}. 

Details of polar and pedal transformations are covered in \cref{app:polar-pedal}. The parameters for a pair of confocal ellipses (or hyperbolas) which are the polar image of the bicentric pair are given in \cref{app:bicentric-to-confocal}. Conversely, the parameters for a bicentric pair which is the polar image of confocal ellipses are given in
\cref{app:confocal-to-bicentric}. In \cref{app:bicentric-vertices-n34} we provide elementary parametrizations for the vertices of $N=3$ and $N=4$ bicentric polygons. In \cref{app:pedal-perimeters-n34} we provide explicit expressions of their perimeters and sums of cosines as well as curious properties thereof.

\subsection*{Related Work}

A few of our experimental conjectures for elliptic billiard N-periodic invariants \cite{reznik2020-intelligencer,garcia2020-new-properties} have been proved: (i) invariant sum of cosines and (ii) invariant product of outer polygon cosines \cite{akopyan2020-invariants,bialy2020-invariants}, and (iii) invariant outer-to-orbit area ratio (for odd N) \cite{caliz2020-area-product}. Dozens of other conjectured invariants appear in \cite{reznik2021-fifty}.

\section{Review: Jacobi's parametrization for bicentric polygons}
\label{sec:jacobi}
In 1828, Jacobi found a beautiful proof for a special case of Poncelet's closure theorem using elliptic functions. In particular, he provided a very simple parametrization for the family of N-sided bicentric polygons that appear in Poncelet's theorem. We will use his parametrization below, and it is appropriate to recall it here.

Referring to Figure~\ref{fig:jacobi-nested}, %and \ref{fig:jacobi-unnested},
consider two circles $\C_{R}$ and $\C_{r}$, with radii $R$ and $r$, respectively. Let $d$ denote the distance between their centers. We will consider polygons that are inscribed in $\C_{R}$ and also are either inscribed or exscribed in $\C_{r}$. By exscribed in $\C_{r}$ we mean that extensions of the sides of the polygon are tangent to $\C_{r}$. Let $p_{j}(u)$, $j=1,...,N$ be the vertices of a N-sided bicentric family of polygons, parametrized by the real variable $u$, with all the vertices in $\C_{R}$.

\begin{figure}
    \centering
    \includegraphics[width=.6\textwidth]{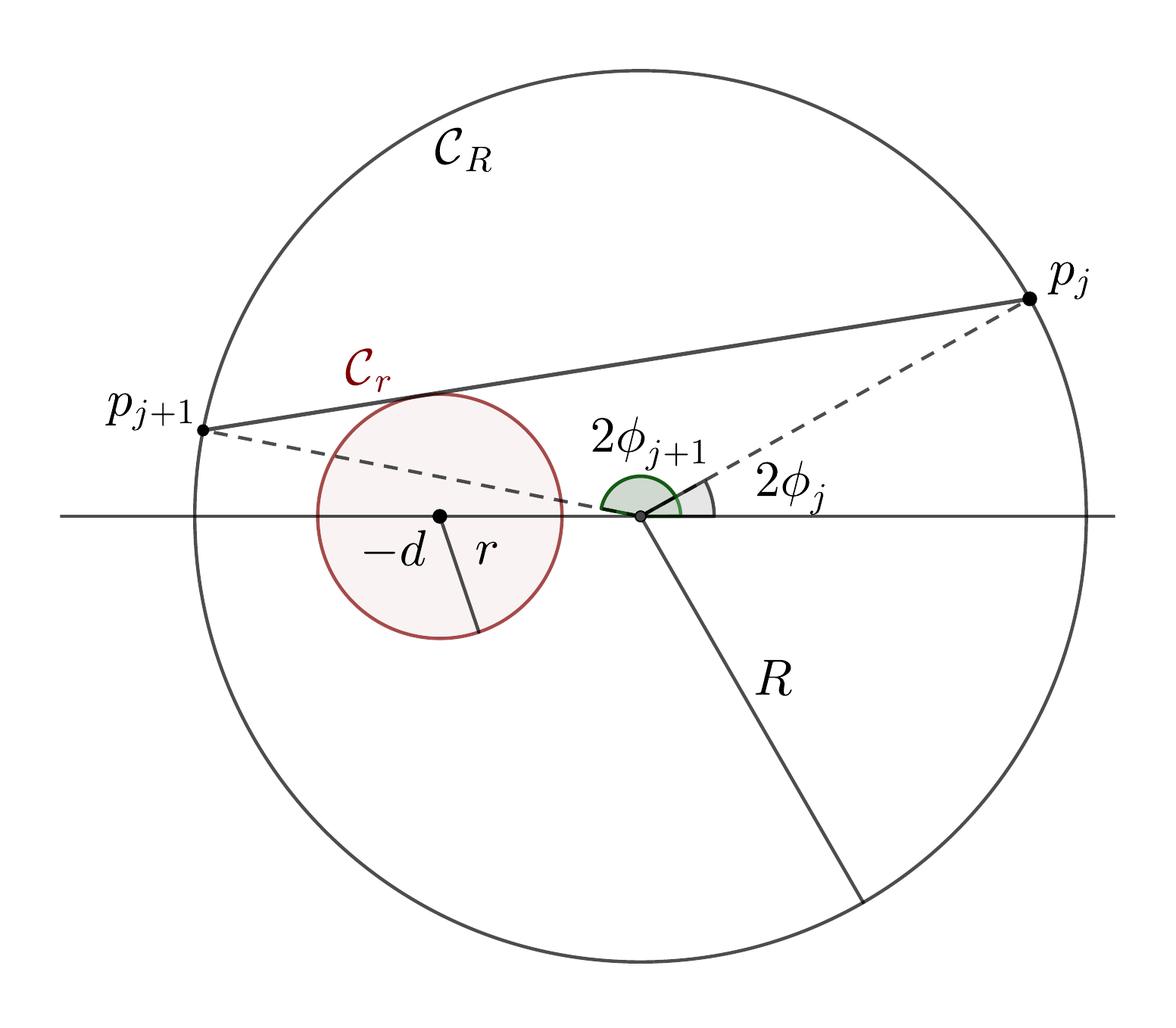}
    \caption{A pair of circles, along with a chord $p_j,p_{j+1}$ of the outer circle tangent to the inner one.}
    \label{fig:jacobi-nested}
\end{figure}

Jacobi noticed that his elliptic functions could be used to provide an explicit expression for the $p_{j}(u)$. Namely, if we write
\begin{equation}
\label{jacobivertex}  
p_{j}(u)=R\left[ \cos{(2\phi_{j}(u))}, \sin{(2\phi_{j}(u))}\right]
\end{equation}

Indeed, he proved that \cite{bos-1987}:

\begin{equation}
\label{jacobiangle}
\phi_{j}(u)=am(u+ j \sigma,k),
\end{equation}
where $am(u,k)$ is the classical Jacobi amplitude function \cite{armitage-2006}, $k$ is the modulus and it is related to $R$, $r$ and $d$ by the following expression \cite[pp. 315]{bos-1987}:

\begin{equation}
\label{jacobirelation}
k^2=\frac{4Rd}{(R+d)^2-r^2},\;\;\;0<k<1
\end{equation}
The real number $K$ is defined by:
\[K=\int_{0}^{\frac{\pi}{2}} \frac{dt}{\sqrt{1-k^2sin^2{t}}}, \]
and finally, $\sigma$ is given by

\[ \sigma=\frac{4\tau K}{N},\]
where $\tau$ is a positive integer and $N>2$.

Actually, Jacobi treated only the case where one of the circles is inscribed, but his argument also holds for the exscribed case \cite{bos-1987}.

Below we recall some fundamental facts about three of Jacobi's elliptic functions: $sn(z,k)=\sin{(am(z,k))}$, $cn(z,k)=\cos(am(z,k))$ and $dn(z,k)=\sqrt{1-k^2sn^2(z,k)}$, where $z \in \mathbb{C}$, and $0<k<1$ is the elliptic modulus. Since $k$ is fixed, we write $sn(z)$ instead of $sn(z,k)$, etc.

These functions have two independent periods and also have simple poles at the same points. In fact:

\begin{align*}
    sn(u+4K)&=sn(u+2iK')=sn(u)\\
    cn(u+4K)&=cn(u+2K+2iK')=cn(u)\\
    dn(u+2K)&=dn(u+4iK')=dn(u)\\
    K'&=K(k'), \;\;k'=\sqrt{1-k^2}
\end{align*}
The  poles of these three functions, which are simple, occur at the points
\[2mK+i(2n+1)K'
,\;\; m,n\in \mathbb{Z}\]

They also display a certain symmetry around the poles. Namely, if $z_p$ is a pole of $sn(z)$, $cn(z)$ and $dn(z)$, then, for every $w \in \mathbb{C}$, we have \cite[Chapter 2]{armitage-2006}:

\begin{align}
sn(z_p+w)=&-sn(z_p-w) \nonumber \\
cn(z_p+w)=&-cn(z_p-w)  \label{eqn:zpole} \\
dn(z_p+w)=&-dn(z_p-w) \nonumber
\end{align}

\section{Bicentric Family: Invariant Sum of Cosines}
\label{sec:bicentric-sum-of-cosines}
\begin{theorem}
The sum of cosines of angles internal to the family of N-periodics interscribed in a bicentric pair is invariant.
\label{thm:bicentric-sum}
\end{theorem}

\begin{proof}
Let $\{p_{j}(u)\}$, as in  \eqref{jacobivertex}, denote the vertices of the family of bicentric polygons. Let $\theta_{j}(u)$ denote the internal angle at the vertex $p_{j}(u)$. It follows from elementary geometry that
$\cos{\theta_{j}(u)}=-\cos({\phi_{j+1}(u)-\phi_{j-1}(u)}).$
Thus, if we denote by $S(u)$ the sum of the cosines of the internal angles, we have:

\begin{equation}
\label{sumofcosines}
S(u)=\sum \cos{\theta_{j}(u)}=-\sum{\left(cn(u_{j+1})cn(u_{j-1})+sn(u_{j+1})sn(u_{j-1})\right)},
\end{equation}
where $u_{j}=u+j\sigma$.

We now consider the natural complexified version of $S(u)$ defined on the complex plane by assuming that $u$ is a complex variable. To prove that $S$ is constant, it is sufficient to show that it has no poles and then apply Liouville's theorem.

So, suppose that $u=u_p$ is a pole of $S$. This implies that, for a certain index $j$, $u_{j}=u_{p}+j\sigma$ is a common pole of $cn(z)$ and $sn(z)$. We will now see that this leads to a contradiction.

In fact, by looking at (\ref{sumofcosines}), the terms where $u_{j}$ appears are given by
$$-(cn(u_{j})cn(u_{j-2})+sn(u_{j})sn(u_{j-2})+cn(u_{j+2})cn(u_{j})+sn(u_{j+2})sn(u_{j})).$$

Thus, the coefficients of $cn(u_{j})$ and $sn(u_{j})$ are, respectively:

$$-(cn(u_{j-2})+cn(u_{j+2})),$$

%\textcolor{red}{pedro}
$$-(sn(u_{j-2})+sn(u_{j+2})).$$

Note that by (\ref{eqn:zpole}) both coefficients are zero, and they cancel out the simple poles of $cn(z)$ and $sn(z)$ at $u_{j}$, so $u_{p}$ is not a pole of $S$.   
\end{proof}

\section{Bicentric Limiting Pedals: Invariant Perimeter}
\label{sec:pedal-perimeter}
In this section we prove that the two pedal polygons of a bicentric Poncelet family with respect to circles centered on either of its two {\em limiting points} (see below) conserve perimeter.

\begin{definition}[Pedal Polygon]
Given a planar polygon $\P$ and a point $p$, the {\em pedal polygon} $\P_{\perp}$ of $\P$ wrt $p$ has vertices $q_{j}$ at the orthogonal projections of $p$ onto the jth sideline $p_{j}p_{j+1}$ or extension thereof.
\end{definition}

\begin{definition}[Limiting Point]
Any pair of non-intersecting circles is associated with a pair of ``limiting'' points $\ell_1,\ell_2$ which lie on the line connecting the centers, with respect to which the circles are inverted to a concentric pair.
\end{definition}

Let $\C_{R}$ be a circle of radius $R$ centered at the origin $(0,0)$ and $\C_{r}$ be a circle of radius $r$ centered at $(-d,0)$. Then the limiting points $(\delta_{\pm},0)$ of the pencil of circles defined by $\C_{R}$ and $\C_{r}$ has abscissa given by \cite[Limiting Point, Eqn. 5]{mw}:

\begin{equation}
\delta_{\pm}=\frac{r^2-R^2-d^2 \pm %\sqrt{R^4-2R^2d^2-2R^2r^2+d^4-2d^2r^2+r^4}}{2d}
\sqrt{d^4 -2(R^2 +r^2)d^2 + (R^2 - r^2)^2 }}{2d}
\label{eqn:limiting-point}
\end{equation}

Let $\P(u)$ be the family of bicentric polygons with respect to a pair of circles $C_{R}$ and $C_{r}$, where $u$ is the real parameter introduced by Jacobi, with vertices given by (\ref{jacobivertex}). Let $\ell$ denote a limiting point of the pencil defined by these circles, as in \eqref{eqn:limiting-point}. Below we derive an expression for the length of the sides of pedal polygons  $\P_{\perp}(u)$ defined by $\P(u)$ and $\ell$.

\begin{lemma}
\label{lem:pedal-perimeter}
 Let $p_{j-1}(u)$, $p_{j}(u)$ and $p_{j+1}(u)$ be three consecutive vertices of $\P(u)$, let $s_{j}(u)=|q_{j+1}(u)- q_{j}(u)|$ be jth sidelength of $\P_{\perp}$. Then:

\begin{equation}
\label{pedalside}
s_{j}(u)=\frac{r_{j}(u)\rho_{j}(u)}{2R},
\end{equation}
where $r_{j}(u)=\left|p_{j-1}(u)-p_{j+1}(u)\right|$ and $\rho_{j}(u)=\left|\ell-p_{j}(u)\right|$. In addition, $r_{j}(u)$ and $\rho_{j}(u)$ are given by the following expressions.
  
\begin{align}
r_{j}(u)=&2R\sin{(\phi_{j+1}(u)-\phi_{j-1}(u))} \label{sideterm} \\
=&2R\left(sn(u_{j+1})cn(u_{j-1})-sn(u_{j-1})cn(u_{j+1})\right) \nonumber
\end{align} 
where $u_{j}=u+j \sigma$.

\begin{equation}
\label{spoketerm}
\rho_{j}(u)=\frac{2}{k}\sqrt{-\delta_{\pm}R}\,\,dn(u_{j}).
\end{equation}
\end{lemma}
\begin{proof}
The proof of \eqref{pedalside} follows the standard one for sidelengths of the pedal triangle \cite[pp. 135--141]{johnson1960}. Equation \eqref{sideterm} follows by inspection from \cref{fig:jacobi-nested}, and the definition of Jacobi's $sn(u)$ and $cn(u)$. Finally, \eqref{spoketerm} is a long but simple computation. Below we show a few intermediate steps. First, if we let $\ell=(\delta_{\pm},0)$ be either limiting point. Then:

\[\rho_{j}(u)=\cos(\phi_{j}(u))\sqrt{R^2-2R\delta_{\pm}}+\delta_{\pm}^2\]

It is straightforward to check that $\delta_{\pm}<0$. Substitute the expression \eqref{eqn:limiting-point} for $\delta_{\pm}$ in the expression for $\rho_{j}(u)$ to obtain:

\[\rho_{j}(u)=\sqrt{\frac{-\delta_{\pm}}{d}\Big( R^2+d^2-r^2-2Rd\cos{(2\phi_{j}(u))}\Big)}\]

Finally, using \eqref{jacobirelation}, we get:

\[\rho_{j}(u)=\frac{2}{k}\sqrt{-\delta_{\pm}R}\sqrt{1-k^2sn^2(u_{j})}=\frac{2}{k}\sqrt{-\delta_{\pm}R}\,\,dn(u_{j})\] 
\end{proof}

We are now in a position to prove the following.

\begin{theorem}
The perimeters $L_\pm$ of the pedal polygons of the bicentric Poncelet family with respect to either limiting point are invariant.
\label{thm:bicentric-pedal-perimeter}
\end{theorem}

\begin{proof}
From Lemma \ref{lem:pedal-perimeter}, the perimeter is given by:

\[ L_\pm(u)=\frac{\sqrt{-\delta_{\pm}R}}{k}\sum{dn(u_{j})\Big(sn(u_{j+1})cn(u_{j-1})-sn(u_{j-1})cn(u_{j+1})\Big)} \]

To prove the above is constant, we consider its natural complexified version, that is, we think of $L_\pm$ as function of a complex variable $u$. Clearly, $L_\pm$ becomes a meromorphic function defined on the complex plane. To prove that $L_\pm$ is constant we will show that it is entire and bounded. So by Liouville's theorem it must be constant.

In turn this amounts to showing $L_\pm$ has no poles. Now, suppose that, for $u=u_p$, a certain $u_{j}$ is a common simple pole of $sn(z)$, $cn(z)$ and $dn(z)$. This is the only way that $L_\pm$ can have a pole.

From the expression of $L_\pm$, it follows there are three terms in the sum where the pole $u_{j}$ of the three Jacobian elliptic functions appears:

\begin{align*}
dn(u_{j-1})&\left(sn(u_{j})cn(u_{j-2})-sn(u_{j-2})cn(u_{j})\right)\\
dn(u_{j})&\left(sn(u_{j+1})cn(u_{j-1})-sn(u_{j-1})cn(u_{j+1})\right)\\
dn(u_{j+1})&\left(sn(u_{j+2})cn(u_{j})-sn(u_{j})cn(u_{j+2})\right)
\end{align*}

We have to prove that the sum of these terms is finite at $u_{j}$. To see this, consider first the term that multiplies $dn(u_{j})$, namely

\[sn(u_{j+1})cn(u_{j-1})-sn(u_{j-1})cn(u_{j+1}).\]

Since $u_j=u+j \sigma$ is a pole and   $u_{j+1}=u_j+ \sigma$, $u_{j-1}=u_j- \sigma$ it follows from \eqref{eqn:zpole} that  $sn(u_{j-1})=-sn(u_{j+1})$ and $cn(u_{j-1})=-cn(u_{j+1})$. Therefore, the expression above is zero. And this cancels the simple pole of $dn(u)$ at $u_{j}$. The same argument can be applied to the terms that multiply $sn(u_{j})$ and $cn(u_{j})$ and this shows that $u_p$ is not a pole of $L_\pm$. 

So $L_\pm$ has no poles and by the periodicity of the elliptic functions, it must be bounded. Thus, by Liouville's theorem $L_\pm$ is constant.
\end{proof}

%In Appendix~\ref{app:explicit-perimeter} we provide explicit %expressions (in terms of Jacobi elliptic functions) for $L_\pm$.

In \cref{app:bicentric-to-confocal,app:confocal-to-bicentric} we show that the image of two nested circles wrt to $\ell_1$ is a confocal pair of ellipses, therefore under this transformation, a bicentric N-gon is sent to an elliptic billiard N-gon. Lemmas \ref{lem:polar} and \ref{lem:limit-focus} found in the Appendix~\ref{app:polar-pedal} show that the bicentric pedal with respect to $\ell_1$ is identical to its polar image (elliptic billiard N-periodic) inverted with respect to a circle centered on $f_1=\ell_1$. Therefore:

\begin{corollary}
Over the family of N-periodics in the elliptic billiard (confocal pair), the perimeter of inversions of said N-periodics with respect to a focus-centered circle is invariant. 
\label{cor:inv-per}
\end{corollary}

Though not yet proved, experimental evidence suggests:

\begin{conjecture}
The sum of cosines of bicentric pedal polygons with respect to either limiting point is invariant, except for the $\ell_1$-pedal in the $N=4$ case.
\label{conj:limiting-sum-cosines}
\end{conjecture}

\section{A Tale of Five Polygons}
\label{sec:five-polys}
Illustrated in \cref{fig:five-polys} is the bicentric family along with its two limiting-pedals and its two polar images (elliptic and hyperbolic billiards), each with respect to a limiting point. While N-periodics in the elliptic billiard conserve perimeter, their hyperbolic version conserve {\em signed} perimeter, i.e., the length of a trajectory segment touching both hyperbola branches (resp. a single branch) is subtracted (resp. added) to the perimeter.

As shown in \cref{fig:bicentric-diagram}, the bicentric family can be regarded as a  ``hub'' from which four derived polygon families can be obtained, all of which conserve both (signed) perimeter and the sum of cosines. Bicentric polygons themselves have variable perimeter.

\begin{figure}
    \centering
    \includegraphics[width=\textwidth]{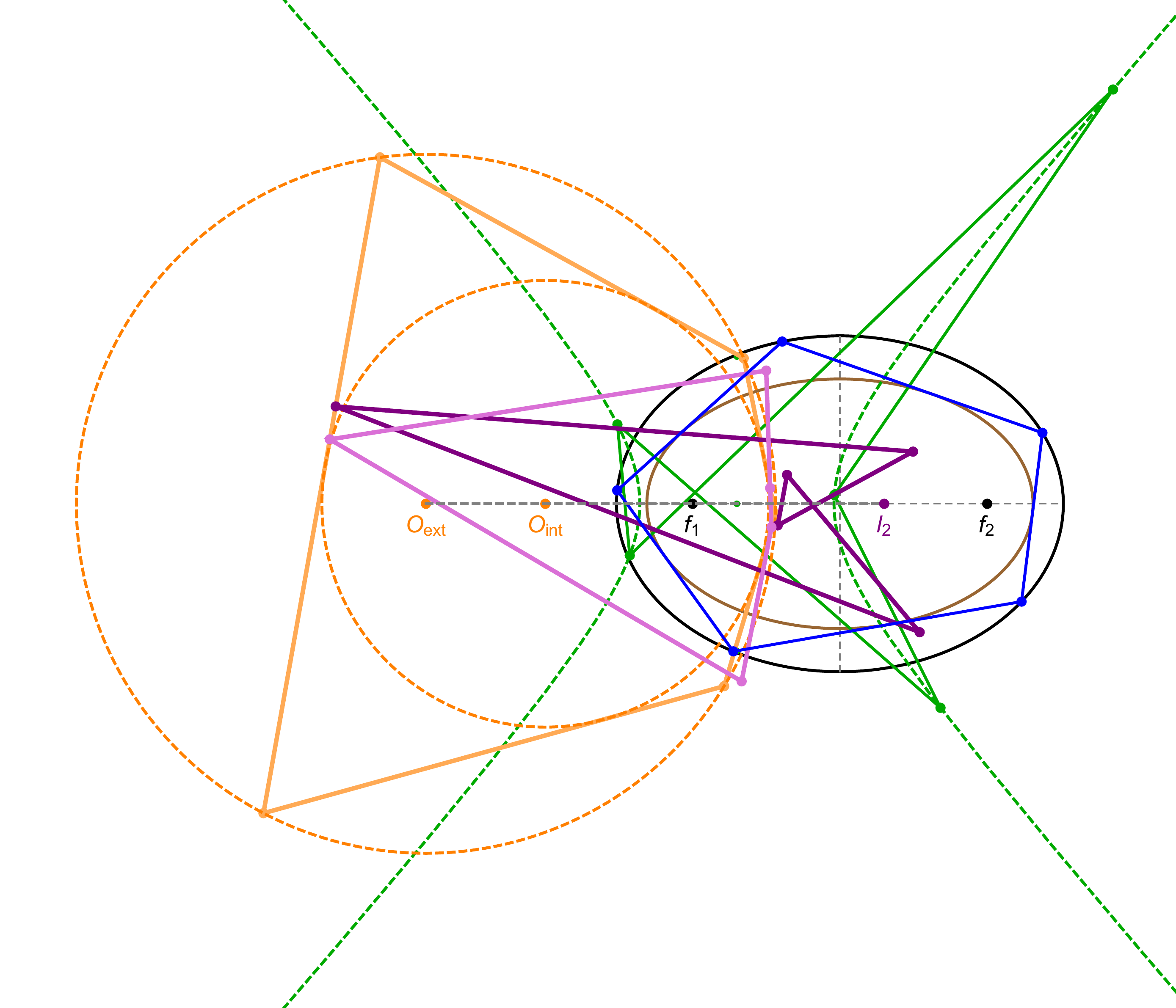}
    \caption{The bicentric family (solid orange) and its two polar images: the elliptic billiard (blue) and the hyperbolic billiard (green). Also shown are the two pedal polygons (pink and purple) with respect to the limiting points $\ell_1=f_1$ and $\ell_2$. All but the bicentrics have constant perimeter. All five families conserve their sum of cosines.}
    \label{fig:five-polys}
\end{figure}

\begin{figure}
    \centering
    \includegraphics[trim=100 50 100 50,clip,width=\textwidth]{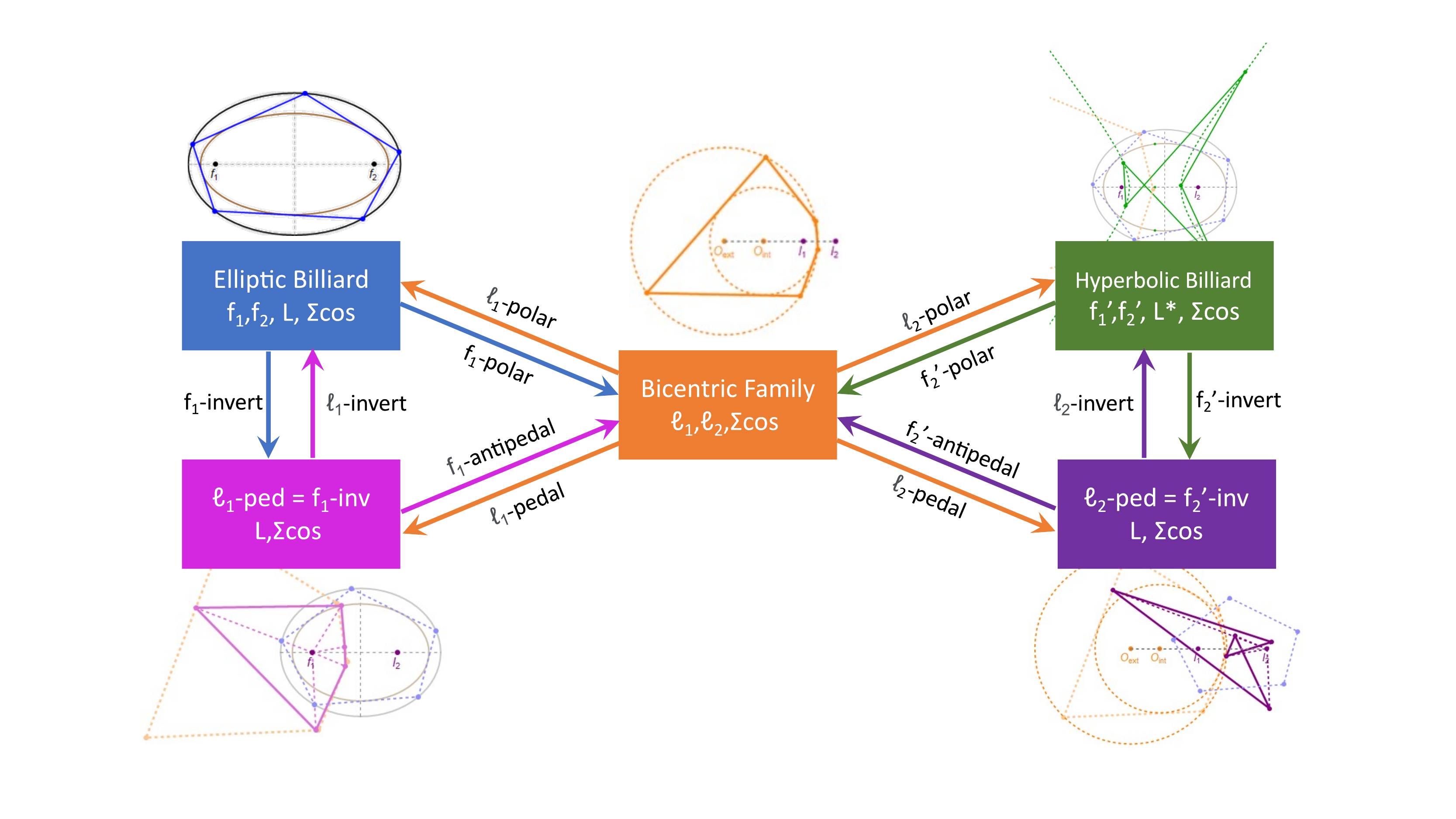}
    \caption{The bicentric family (orange, center) is the hub from which four polygon families can be derived: the (i) elliptic (resp. (ii) hyperbolic) billiard with foci $f_1=\ell_1,f_2$ (resp. $f_1'$ and $f_2'=\ell_2$) is the bicentric polar image with respect to $\ell_1$ (resp. $\ell_2$); (iii) the first (resp. second) pedal is obtained with respect to $\ell_1$ (resp. $\ell_2$). These are the inversive image of the elliptic and hyperbolic billiards with respect to $l_1=f_1$ or $l_2=f_2'$, respectively. With the exception of the bicentrics, all 4 derived families conserve perimeter $L$ (in the case of the hyperbolic billiard it is the {\em signed} perimeter $L^*$ which is conserved). All 5 families conserve sum of cosines, except for $N=4$ $\ell_1$-pedals.}
    \label{fig:bicentric-diagram}
\end{figure}

\section{List of Videos}
\label{sec:videos}
Videos illustrating some of the above phenomena are listed on Table~\ref{tab:playlist}.

\begin{table}
\small
\begin{tabular}{|c|c|l|l|}
\hline
id & N & Title & \textbf{youtu.be/<.>}\\
\hline
01 & 5 & {Invariant-perimeter limiting point pedals} &
\href{https://youtu.be/8m21fCz8eX4}{\texttt{8m21fCz8eX4}}\\
02 & 3...8 & {Bicentric Pedals, Polars, and Inversions I} &
\href{https://youtu.be/jhXDKRFLpVk}{\texttt{jhXDKRFLpVk}}\\
03 & 3...8 & {Bicentric Pedals, Polars, and Inversions II} &
\href{https://youtu.be/A7F3szW7rUE}{\texttt{A7F3szW7rUE}}\\
04 & 3...8 & {Bicentric Pedals, Polars, and Inversions III} &
\href{https://youtu.be/6TmaezNFrOs}{\texttt{6TmaezNFrOs}}\\
05 & 4 & {Bicentric Pedals, Polars, and Inversions IV} &
\href{https://youtu.be/fZe6elRTfeA}{\texttt{fZe6elRTfeA}}\\
06 & 5 & \makecell[lt]{A Rose in the Elliptic Garden:\\the Invariant-Perimeter, Focus-Inversive Family} &
\href{https://youtu.be/wkstGKq5jOo}{\texttt{wkstGKq5jOo}}\\
07 & 5 & \makecell[lt]{Focus-inversive polygons of elliptic Billiard\\self-intersected 5-periodics} &
\href{https://youtu.be/LuLtbwkfSbc}{\texttt{LuLtbwkfSbc}}\\
08 & 5 & \makecell[lt]{Inversive arcs of elliptic billiard N-periodic\\segments and the bicentric family} &
\href{https://youtu.be/mXkk_4RYrnU}{\texttt{mXkk\_4RYrnU}}\\
09 &  6 & \makecell[lt]{Focus-inversive polygons of elliptic Billiard\\self-intersected 6-periodics} &
\href{https://youtu.be/7lXwjXj-8YY}{\texttt{7lXwjXj-8YY}}\\
10 & 7 & \makecell[lt]{Focus-inversive polygons of elliptic billiard\\self-intersected 7-periodics I} &
\href{https://youtu.be/BRQ39O9ogNE}{\texttt{BRQ39O9ogNE}}\\
11 & 7 & \makecell[lt]{Focus-inversive polygons of elliptic billiard\\self-intersected 7-periodics II} &
\href{https://youtu.be/gf_aHyvbqOY}{\texttt{gf\_aHyvbqOY}}\\
12 & 8 & \makecell[lt]{Focus-inversive polygons of elliptic billiard\\self-intersected 8-periodics I} &
\href{https://youtu.be/5Lt9atsZhRs}{\texttt{5Lt9atsZhRs}}\\
13 & 8 & \makecell[lt]{Focus-inversive polygons of elliptic billiard\\self-intersected 8-periodics II}&
\href{https://youtu.be/93xpGnDxyi0}{\texttt{93xpGnDxyi0}}\\
\hline
14* & 5 & \makecell[lt]{Circular loci of focus-inversive centroids\\of elliptic billiard simple 5-periodics} &
\href{https://youtu.be/hjzW84ZZApA}{\texttt{jzW84ZZApA}}\\
15* & 5 &\makecell[lt]{Circular loci of focus-inversive centroids\\of elliptic billiard self-intersected 5-periodics} &
\href{https://youtu.be/7bzID9SVwqM}{\texttt{7bzID9SVwqM}}\\
16* & 6 & \makecell[lt]{Focus-inversive polygons of elliptic billiard\\6-periodics and the null-area antipedal polygon} & 
\href{https://youtu.be/fOAES-CzjNI}{\texttt{fOAES-CzjNI}}\\
17* & 5 & \makecell[lt]{Invariant area ratio of focus-inversives\\to elliptic billiard N-periodics} & \href{https://youtu.be/eG4UCgMkKl8}{\texttt{eG4UCgMkKl8}}\\
18* & 5 & \makecell[lt]{Invariant product of areas amongst the two\\focus-inversive elliptic billiard N-periodics (odd N)} & \href{https://youtu.be/bTkbdEPNUOY}{\texttt{bTkbdEPNUOY}}\\
\hline
\end{tabular}
\caption{Videos illustrating some phenomena presented herein. The last column is clickable and provides the YouTube code. The entries whose id has an asterisk (*) represent phenomena detected experimentally but not yet proved.}
\label{tab:playlist}
\end{table}

\section*{Acknowledgements}
\noindent We would like to thank Arseniy Akopyan, Sergei Tabachnikov, and Jair Koiller for invaluable discussions during the discovery phase. The second author is fellow of CNPq and coordinator of Project PRONEX/ CNPq/ FAPEG 2017 10 26 7000 508.

\appendix

\section{Polar Pedal Transformations}
\label{app:polar-pedal}

We review properties of polar and pedal transformations. A detailed treatment is found in \cite{akopyan2007-conics,stachel2019-conics}.

In the discussion that follows, all geometric objects are contained in a fixed plane. Let $\C$ be a circle centered at $f_1$. The polar transformation with respect to $\C$ maps each straight line not passing through $f_1$ into a point, and maps each point different from $f_1$ into a straight line. This is done in the following manner.

 Let $p\neq f_{1}$ be a point and let ${p}^\dagger$ be the inversion of $p$ with respect to $\C$. The straight line $L_{p}$ that passes through ${p}^\dagger$ and is orthogonal to the line joining $p$ and ${p}^\dagger$ is the polar of $p$ with respect to $\C$. Conversely, a line $L$ not passing through $f_1$ has a point $q$ as its pole with respect to $\C$ if $L=L_{q}$.

For a smooth curve $\gamma$ not passing through $f_1$, we can define the polar curve $\gamma^{\star}$ in two equivalent ways. Let $p$ be a point of $\gamma$ and $T_{p}\gamma$ the tangent line to $\gamma$ at $p$, we define $p^{\star}=(T_{p}\gamma)^{\star}$, and $\gamma^{\star}$ is the curve generated by $p^{\star}$ as $p$ varies along $\gamma$. We can also think of $\gamma^{\star}$ as the curve that is the envelope of the 1-parameter family of lines $L_{p}$, where $p$ is a point of $\gamma$.

The notion of a polar curve can be naturally extended to polygons in the following manner: let $L_{j}$, $j=1,2,...,N$ be the consecutive sides of a planar polygon $\P$, and let $q_{j}$ be the corresponding poles, then this indexed set of points are the vertices of what we call the polar polygon $\P^{\star}$. Alternatively, we can consider the polars of vertices of $\P$, and their consecutive intersections do define the vertices of $\P^{\star}$.

Although the next results are certainly classical, we couldn't find them explicitly in the literature, so we include them for the reader's convenience.

\begin{lemma}
Let $\E$ be an ellipse and $f_1$ one its the foci. Then the polar curve $\E^{\star}$ with respect to a circle $\C$ centered at $f_1$ is a circle. Let $\H$ be a hyperbola and $f_1$ one of its foci. Then the polar curve $\H^{\star}$ with respect to a circle $\C$ centered at $f_1$ is a circle minus two points.
\label{lem:polar}
\end{lemma}
\begin{proof}

We will use polar coordinates for our computations. Without loss of generality, let $f_{1}=(0,0)$ and consider the parametrized conic given by:

\[ \gamma(t)=\left[\frac{a(1-e^2)}{1+e\cos{t}}\cos{t}, \frac{a(1-e^2)}{1+e\cos{t}}\sin{t}\right] \]

\noindent where, if $e>1$ the trace of $\gamma$ is a hyperbola and if $e<1$ the trace of $\gamma$ is an ellipse. The expression for the polar curve $\gamma^{\star}(t)$ is obtained by direct computation: compute the unit normal $\mathbf{n}(t)$ to $\gamma(t)$, and  the distance $d(t)$ from the tangent line through $\gamma(t)$ to $f_{1}$. This yields:
 
\[ \gamma^{\star}(t)=\left[\frac{e+\cos{t}}{a(1-e^2)},\frac{\sin{t}}{a(1-e^2)}\right] \]

\noindent whose trace is clearly contained in a circle. For the hyperbola, the parameter $t$ is such that $1+e\cos(t) \neq 0 $, this is why $\H^{\star}$ is a circle minus two points.
  
\end{proof}

\begin{lemma}
\label{lem:limit-focus}
Let $\E_{1}$ and $\E_{2}$ be two confocal ellipses and $\E_{1}^{\star}$ and $\E_{2}^{\star}$ be the circles as in Lemma \ref{lem:polar}, then $f_{1}$ is a limiting point of the pencil of circles defined by $\E_{1}^{\star}$ and $\E_{2}^{\star}$. In a similar way, let $\E$ and $\H$ be respectively an ellipse and hyperbola that are confocal, and let $\E^{\star}$ and $\H^{\star}$ be the circle and the circle minus 2 points, as in Lemma \ref{lem:polar}, then $f_{1}$ is a limiting point of the pencil of circles defined by $\E^{\star}$ and the circle that contains $\H^{\star}$.
\end{lemma}

\begin{proof}
Given two circles $\C_{1}$ and $\C_{2}$, a classical result states that the limiting points $\delta_{\pm}$ of the pencil of circles determined by $\C_{1}$ and $\C_{2}$ are such that the inversion of $\C_{1}$ and $\C_{2}$ with respect to circles centered on $\delta_{\pm}$ are concentric.

If we denote by $a_{i}$ and $e_{i}$, $i=1,2$, respectively, the semi-major diameter and eccentricity of the ellipses $\E_{1}$ and $\E_{2}$, then, by symmetry, we can define an unknown limiting point as $\delta_{p}=(x,0)$, and the concentric circle condition then becomes a quadratic equation in the variable $x$, where the coefficients depend on $a_{1}$, $a_{2}$, $e_{1}$ and $e_{2}$. Using the fact that $\E_{1}$ and $\E_{2}$ are confocal, which is equivalent to $e_{1}a_{1}=e_{2}a_{2}$, and with some algebraic manipulations, the quadratic equation can be written as:

\[ (a_{1}-a_{2})(a_{1}+a_{2})(e_{2}a_{2}x+1)=0 \]

\noindent so $f_{1}=(0,0)$ is indeed one of the limiting points of the pencil of circles.
\end{proof}

\section{Polar Image of Bicentric Pair}
\label{app:bicentric-to-confocal}
Consider the pair of nested circles:
 
\[\C_{int}:\;x^2+y^2=r^2,\;\;\;\C_{ext}:\;(x+d)^2+y^2=R^2\]
Their limiting points $\ell_1$ and $\ell_2$ are given by \cite[Limiting Points]{mw}:
 
 \[ \ell_1= (R^2 - d^2 - r^2 -\Delta)/(2d),\;\;\;\ell_2=(R^2 - d^2 - r^2 + \Delta)/(2 d)\]
where:
 \[ \Delta=\sqrt { \left( d+R+r \right)  \left( R-d+r \right)  \left( R+d-r
 \right)  \left( R-d-r \right) }\]

Notice $\ell_1$ (resp. $\ell_2$) is internal (resp. external) to the circle pair. Below we show that the polar image of the  $\C_{int},\C_{ext}$ pair with respect to a circle of radius $\rho$ centered on $\ell_1$ (resp. $\ell_2$) is a confocal pair of ellipses (resp. hyperbolas).

\begin{lemma}
The polar image of $\C_{int}$ with respect to $\ell_1 $ is the ellipse $\E$ centered at
 \[ \O_e=\left[  {\rho}^2\frac{d}{\Delta} +\frac{k-\Delta}{2d},0\right] \]
where $k=R^2 - d^2 - r^2$. Its semi-axes are given by:
 
\[  a^2={\rho}^4\left(\frac{2d^2r^2+ \Delta(k+\Delta)}{ 2 \Delta^2 r^2}\right),\;\;\;b^2= {\rho}^4\left(\frac{k+ \Delta}{2 \Delta r^2}\right) \]
 \end{lemma}
\noindent Note that $c^2=a^2-b^2={\rho}^4d^2/{\Delta^2}$.
 
 \begin{lemma}
 The polar image of $\C_{ext}$ with respect to $\ell_1$ is an ellipse $\E'$ confocal with $\E$ with semi-axes given by:
 
\[  a'^2= {\rho}^4\frac{(2R^2d^2 + \Delta (k'+\Delta ))}{ 2 \Delta^2 R^2},\;\;\;b'^2=  {\rho}^4\left(\frac{k'+\Delta}{ 2 \Delta R^2}\right) \]
 \end{lemma}
 \noindent where $k'=R^2 + d^2 - r^2$.
 
 \begin{lemma}
 The polar image of $\C_{int}$ with respect to $\ell_2 $ is the hyperbola $\H$ centered at
  \[  \O_h=\left[  {-\rho}^2\frac{d}{\Delta} +\frac{k+\Delta}{2d},0\right] \]
 with semiaxes given by:
 
\[  a_h^2={\rho}^4\left(\frac{  2d^2r^2-\Delta(k - \Delta)}{2  \Delta^2 r^2}\right),\;\;\;b_h^2= {\rho}^4\left(\frac{k - \Delta}{2 \Delta r^2}\right)\]
 \end{lemma}
 \noindent Note that $c_h^2=a_h^2+b_h^2={{\rho}^4d^2}/{\Delta^2}$. Note also that $c=c_h$.
 
 \begin{lemma} The polar image of $\C_{out}$ with respect to $\ell_2$ is a hyperbola $\H'$ confocal with $\H$. Its semiaxes are given by:
 
\[ a_h'^2= {\rho}^4\left(\frac{2R^2d^2 - \Delta (k' - \Delta)}{ 2 \Delta^2 R^2}\right),\;\;\; b_h'^2= {\rho}^4 \left(\frac{k' - \Delta}{ 2 \Delta R^2}\right)\]
 \end{lemma}

\section{Polar Image of Confocal Pair}
\label{app:confocal-to-bicentric}
Consider a pair of origin-centered confocal ellipses $\E$ and $\E'$ with semi-axes $a,b$ and $a',b'$, respectively. Their common foci $f_1,f_2$ lie at:

\[ f_1=[-c,0],\;\;\;f_2=[c,0] \]
where $c^2=a^2-b^2$.

Below we show that the polar image of the $\E,\E'$ pair with respect to a circle of radius $\rho$ centered on $f_1$ is a pair of nested circles $\C_{int},\C_{ext}$ with centers given by:

\[\O_{int}=[-c-\rho^2\frac{c}{b^2},0],\;\;\;\O_{ext}=[-c-\rho^2\frac{c}{b'^2},0]\]
Note the distance $d$ between said centers is given by: 

\[ d=\rho^2\frac{ c\, (a^2 - {a'}^2)}{b^2\, {b'}^2}=\rho^2  \frac{ca^2J^2}{b'^2}\]
where 
$  J = \sqrt{a^2 - a'^2}/(ab).
$

Their respective radii $r,R$ are given by: 

\[ r=\rho^2\frac{a}{b^2},\;\;\;R=\rho^2\frac{a'}{b'^2}\]

%\begin{align*} 
%  k^2&=\frac{4Rd}{(R+d)^2-r^2}=\frac{4 \,c\, {a'}\,({a'} - c)^2    }{{b'}^4}\\
%  \delta_{\pm}&=\pm {\frac {\rho^{6}{a}^{4}}{2\,c\, {b}^{8}}}+{\frac { \left( {a}^{2} {b'}^{4}-{a'}^{2}{b}^{4}-{c}^{6} \right) \rho^{2}}{2\,{b}^{2}  {b'}^{2} {c}^{3}}}
%\end{align*}

 Let $\ell_1$ (resp. $\ell_2$) be the limiting point internal (resp. external) to $\C_{int},\C_{ext}$.
 
 \begin{lemma}
 The limiting points $\ell_1,\ell_2$ are given by: $[-c,0]$ and $[-c+\frac{\rho^2}{c},0]$.
 \end{lemma}

\section{Bicentric Vertices: N=3,4}
\label{app:bicentric-vertices-n34}
Consider a pair of circles
\[\mathcal{C}_1: x^2+y^2-r^2=0, \;\;\; \mathcal{C}_2: (x+d)^2+y^2-R^2=0\]

\subsection{N=3}

Let $(x_0,y_0)=(r\cos t,r\sin t) \in \mathcal{C}_1$. Let
$d^2=R(R-2r)$. Then the 3-periodic orbit is parametrized by
$\{P_1,P_2,P_3 \}, $ where
{\small  
\begin{align*}
P_1&=\left[{\frac {\cos t(2\,d R    \cos  t    +    {R}^{2}-  {d}^{2})+\Delta\,\sin t}{2\,R}}-d, {\frac {\sin t( 2d R\, \cos
 t  +  {R}^{2}-d^2) -\Delta\,\cos t }{2\,R}}\right]
\\
P_2&=\left[ {\frac { \cos t(2\,d R  \cos t     +{R}^{2}- {d}^{2})-\Delta\,\sin
 t }{2\,R}}-d,{\frac {\sin t(2d R\,  \cos
 t  +   {R}^{2}-  {d}^{2})+\Delta\,\cos t }{2\,R}}\right]
\\
 P_3&=\left[-{\frac { \left( R\cos t   -d \right)  \left( {R}^{2}-{d
}^{2} \right) }{     R^2+d^2-2\,d R\cos t }   }, {\frac 
{ -R\left(   R^2-d^2 \right) \sin t }{ R^2+d^2-2\,d R \cos
 t  }}\right]\\
 \Delta&=\sqrt { \left(  R^2+d^2-2d R\,\cos t   \right) 
 \left(   3\, R^2-d^2 +2\,d R\cos t\right) }
 \end{align*}
 }
Under the above pair of circles, the limiting points are at:

\begin{align*}
    l_1&=\left[\frac{  R^2-d^2 }{8\,d {R}^{2} }
 \left( \sqrt {   (9\,R^2-d^2)    
  (R^2-d^2)     }+3\, R^2+d^2 \right) 
,0\right]\\
l_2&=l_1-\left[\frac{\sqrt{9R^2 - d^2} (R^2-d^2)^{\frac{3}{2}}}{4R^2d}, 0\right].
\end{align*}
 
\subsection{N=4}

Let $(x_0,y_0)\in \mathcal{C}_1$. 
The Cayley condition for a pair of circles to admit Poncelet 4-periodics due to Kerawala is \cite[Poncelet's Porism, Eq. 39]{mw}:

\[ \frac{1}{(R-d)^2}+ \frac{1}{(R+d)^2}-\frac{1}{r^2}=0 \]
  
Let $P_i=[x_i,y_i]$, $i=1,...,4$ denote the vertices of a bicentric 4-periodic. Let $\alpha=R^2+d^2$ and $\beta=R^2-d^2$. The vertices are parametrized as:
 
\begin{align*}
\small
   x_{1}&=\Delta\,y_0-{\frac { \left( \beta+2\,d x_0
 \right)  \left( d \alpha  -\beta x_0 \right) }{2\,\alpha}},\\
 y_{1}&=-\Delta\,x_0+{\frac { \left( 2\,d \beta x_0+
 \alpha ^{2} \right) y_0}{2\,\alpha}}\\
x_{2}&=-\Delta\,y_0-{\frac { \left( \beta+2\,dx_0
 \right)  \left( d \beta - \beta x_0 \right) }{2\,\alpha}}\\
y_{2}&=\Delta\,x_0+{\frac {2\,d \beta y_0\,x_0+
 \alpha ^{2}y_0}{2\,\alpha}}\\
 x_3&= \frac{  ((( x_0^2 \alpha \beta - 4  x_0^2 \alpha^2 + 3/2 \beta^3 - 2 \beta^2 \alpha) \sqrt{2 \alpha - 2 \beta} + \beta (2 \Delta \alpha  y_0 + 8 \alpha^2  x_0 - 8 \alpha \beta  x_0 + \beta^2  x_0)) \beta)}{(4 (\sqrt{2 \alpha - 2 \beta} \alpha  x_0 + \beta (\beta - 2 \alpha)/2)^2)}\\
 y_3&=\frac{\alpha \beta (\alpha (2  x_0  y_0 \alpha + \beta ( x_0  y_0 - 2 \Delta)) \sqrt{2 \alpha - 2 \beta} + (4 \Delta  x_0 - 2 \beta  y_0) \alpha^2 - 2 \Delta \alpha \beta  x_0 + \beta^3  y_0)}{(2 \sqrt{2 \alpha - 2 \beta} \alpha  x_0 - 2 \alpha \beta + \beta^2)^2}\\
 x_4&=-\frac{((( x_0^2 \alpha \beta - 4  x_0^2 \alpha^2 + 3/2 \beta^3 - 2 \beta^2 \alpha) \sqrt{2 \alpha - 2 \beta} + \beta (-2 \Delta \alpha  y_0 + 8 \alpha^2  x_0 - 8 \alpha \beta  x_0 + \beta^2  x_0)) \beta)}{(4 (\sqrt{2 \alpha - 2 \beta} \alpha  x_0 + \beta (\beta - 2 \alpha)/2)^2)}\\
 y_4&=\frac{(\alpha (2  x_0  y_0 \alpha + \beta ( x_0  y_0 + 2 \Delta)) \sqrt{2 \alpha - 2 \beta} + (-4 \Delta  x_0 - 2 \beta  y_0) \alpha^2 + 2 \Delta \alpha \beta  x_0 + \beta^3  y_0) \beta}{(2 \sqrt{2 \alpha - 2 \beta} \alpha  x_0 - 2 \alpha \beta + \beta^2)^2}\\
 \Delta&=\frac{\sqrt{-2\sqrt{2\alpha - 2\beta}\; \alpha x_0\beta^2 - 2 x_0^2\alpha^3 + 2 x_0^2\alpha^2\beta + \alpha^2\beta^2 + \alpha\beta^3 - \beta^4}}{ (2\alpha}
\end{align*}

Under the above pair of circles, the limiting points are at:
\begin{align*}
    l_1=\left[\frac{R^2 - d^2}{2d}, 0\right],\;\;\; l_2=  \left[\frac{d(R^2-d^2)}{R^2 + d^2}, 0\right]
\end{align*}

\section{Limiting Pedal Perimeters for N=3 and N=4}
\label{app:pedal-perimeters-n34}
Below we consider 3- and 4-periodics in the confocal pair where $a,b$ are the semi-axes of the outer ellipse has axes $(a,b)$. Below, set $\delta=\sqrt{a^4-a^2 b^2+b^4}$ and $c^2=a^2-b^2$.

\subsection{N=3 case}

\begin{figure}
    \centering
    \includegraphics[width=.9\textwidth]{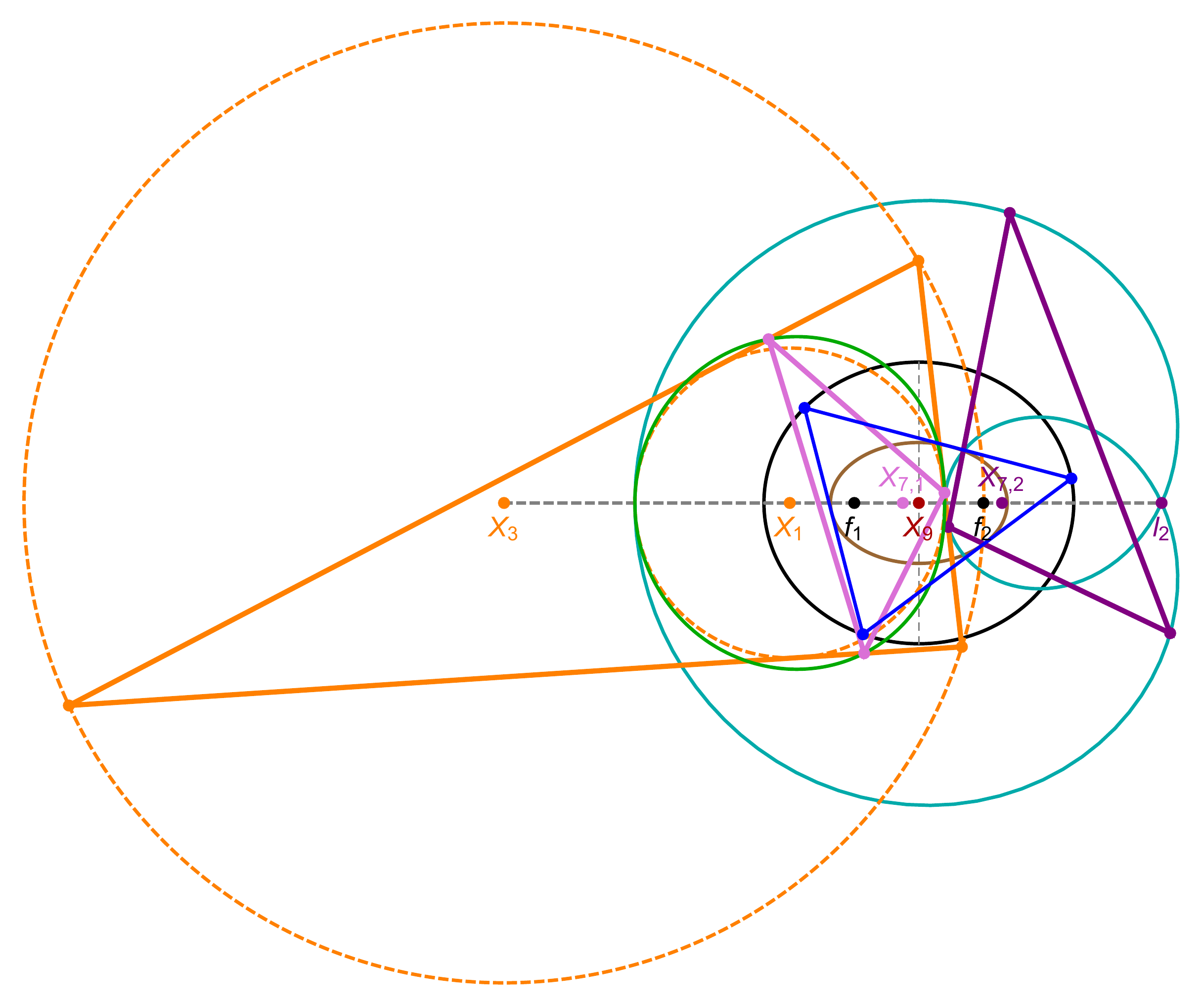}
    \caption{N=3 case: the bicentric family (solid orange) is the poristic family \cite{gallatly1914-geometry}. Its sum of cosines is invariant and equal to those of the two limit point pedals (pink and purple). The Gergonne points $X_{7,1}$ and $X_{7,2}$ of each pedal are stationary. \href{https://bit.ly/379HU8l}{live}}
    \label{fig:n3}
\end{figure}

Referring to Figure~\ref{fig:n3}, the perimeter $L^\dagger$ of the inversive polygon for the $N=3$ family, originally derived in  \cite[Prop. 4]{reznik2020-n3-focus-inversive} is given by:

\[L^\dagger=L_+=\rho^2 \frac {\sqrt { \left( 8\,{a}^{4}+4\,{a}^{2}{b}^{2}+2\,{b}^{4}
 \right) \delta+8\,{a}^{6}+3\,{a}^{2}{b}^{4}+2\,{b}^{6}}}{{a}^{2}{b}^{
2}}\]

By Corollary~\ref{cor:inv-per}, this is equal to the perimeter $L_-$ of the bicentric pedal with respect to the focal limiting point.

\begin{align*}
    L_{-} &= {\frac {  \left( 9\,{R}^{2}-{d}^{2} \right)  \left( {R}^{2}-{d}^{2}
 \right) \sqrt {2}\rho^2}{16\,{R}^{4}d}\sqrt {- \left( {R}^{2}-{d}^{2}
 \right) ^{{\frac{3}{2}}}\sqrt {9\,{R}^{2}-{d}^{2}}+ 3R^4 + 6R^2d^2 - d^4}}\\
R&=(2a^4 - 2a^2b^2 + b^4 + (2a^2 -b^2)\delta  )a\rho^2/b^6,\;\;
d=(2a^2 - b^2 + 2\delta)c\rho^2a^2/b^6
\end{align*}

The perimeter $L_+$ of the bicentric pair with respect to the non-focal limiting point is given by:

\begin{align*}
    L_{+} &= {\frac { \left( 9\,{R}^{2}-{d}^{2} \right)  \left( {R}^{2}-{d}^{2}
 \right) \sqrt {2}\rho^2}{16\,{R}^{4}d}\sqrt { \left( {R}^{2}-{d}^{2}
 \right) ^{{\frac{3}{2}}}\sqrt {9\,{R}^{2}-{d}^{2}}+ 3R^4 + 6R^2d^2 - d^4}}\\
R&=(2a^4 - 2a^2b^2 + b^4 + (2a^2 -b^2)\delta  )a\rho^2/b^6,\;\;
d=(2a^2 - b^2 + 2\delta)c\rho^2a^2/b^6
\end{align*}

The sum of cosines of a triangle is given by $1+r/R$ and is therefore constant for the $N=3$ bicentric family. Let $\theta'_i$ denote the angles of the bicentric polygon. The sum of its cosines can be derived as:

\begin{equation}
\sum\cos\theta' = 1+\frac{r}{R} = \frac{3R^2 - d^2}{2R^2}
\label{eq:bic-cos}
\end{equation}

\begin{proposition}
The sum of cosines for the first and second $N=3$ bicentric pedals are constant and identical to \eqref{eq:bic-cos}.
\end{proposition}

Note: in terms of the associated elliptic billiard parameters, this is given by \cite[Prop. 6]{reznik2020-n3-focus-inversive}:

\[\sum\cos{\theta^\dagger}_{(N=3)}=\frac{\delta (a^2+c^2-\delta)}{a^2c^2} \]

\begin{proof} Using CAS, it follows from  straightforward calculations with the orbit parametrized in Appendix \ref{app:bicentric-vertices-n34}.
\end{proof}

The two limiting pedals have stationary Gergonne points $X_7$. The first one was derived in \cite[Proposition 1]{garcia2020-self-intersected}:

\[ X_{7,1}=\left[c\left(1-\frac{\rho^2}{\delta+c^2}\right),0\right]  \]

\[X_{7,1} =\frac{ (R^2 - d^2)((R^2 - d^2)^{3/2}\sqrt{9R^2 - d^2} + 3R^4 + 6R^2d^2 - d^4)}{16R^4d}\]

%\[X_{7,2}=-\frac{(R^2  -d^2) ((R^2 - %d^2)^{3/2}\sqrt{9R^2 - d^2} - 3R^4 - 6R^2d^2 + d^4)}{16d R^4}
%\]
\subsection{N=4 case}

\begin{figure}
    \centering
    \includegraphics[width=.9\textwidth]{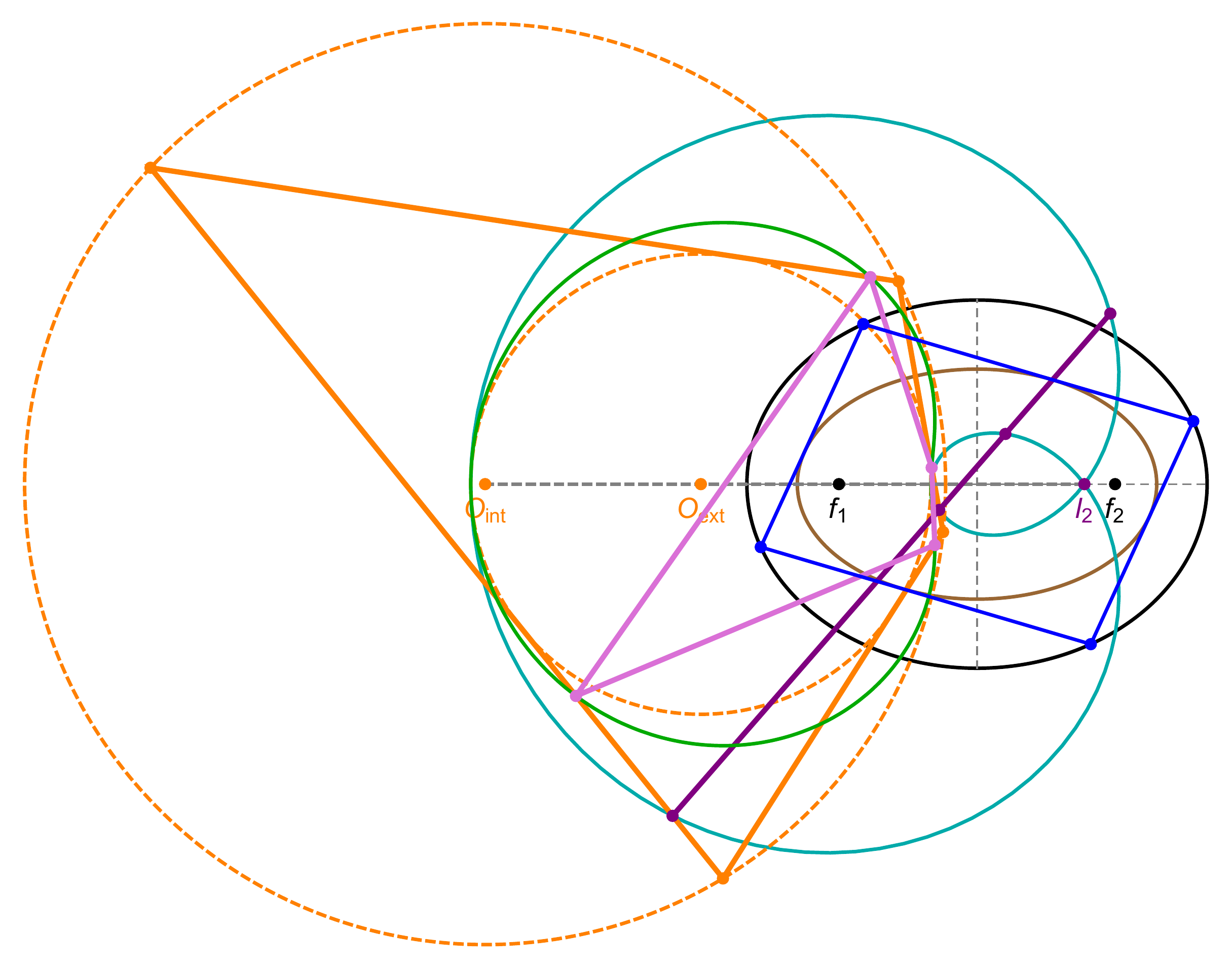}
    \caption{In the N=4 case, remarkable things happen: (i) the sum of cosines of the $f_1$-pedal (pink) is not constant; (iii) its perimeter is the same as the corresponding billiard 4-periodic; (iv) the vertices of the $l_2$-pedal (purple) are collinear and (v) the sum of its cosines is 4. \href{https://youtu.be/fZe6elRTfeA}{Video}}
    \label{fig:n4}
\end{figure}

Referring to Figure~\ref{fig:n4}, the perimeter $L^\dagger$ of the inversive polygon for billiard 4-periodics was originally derived in  \cite[Prop. 18]{garcia2020-self-intersected}. It is identical to the perimeter of 4-periodics themselves and given by:

\begin{equation}
    L^\dagger= L_{+,N=4}= \,{\frac {4\sqrt {{a}^{2}+{b}^{2}}}{{b}^{2}}}
\label{eqn:inv-per-n4} 
\end{equation}

\begin{proposition}
In the $N=4$ family, the vertices of the bicentric pedal with respect to the non-focal limiting point are collinear.
\end{proposition}

\begin{proof}
The polar image of the bicentric family with respect to $\ell_2$ is a pair of confocal hyperbolas, see \cref{app:bicentric-to-confocal}, i.e., the polar image of bicentric 4-periodics is a billiard family. It can be shown its vertices are concyclic with the two hyperbolic foci $f_1',f_2'$, one of which coincides with $\ell_2$. Therefore, the inversion of said vertices with respect to $\ell_2$ is a set of collinear points.
\end{proof}

As before, Equation~\ref{eqn:inv-per-n4} is the same as the perimeter of the first bicentric pedal. The perimeter $L_+$ of the non-focal bicentric pedal is given by:

\[ L_{-,N=4} = \frac{4a^2}{b^2 c} \]

Regarding the sum of cosines, it is well-known a circle-inscribed quadrilateral has supplementary opposing angles, i.e.:

\begin{observation}
The sum of cosines of a bicentric N=4 family is null.
\end{observation}

Since the second bicentric pedal is a degenerate polygon:

\begin{observation}
The sum of cosines of the second limiting pedal to the N=4 bicentric family is equal to 4.
\end{observation}

\bibliographystyle{maa}
\bibliography{references,authors_rgk_v3}

\begin{thebibliography}{10}
\expandafter\ifx\csname urlstyle\endcsname\relax
 \providecommand{\url}[1]{doi:\discretionary{}{}{}#1}\else
 \providecommand{\url}{doi:\discretionary{}{}{}\begingroup
  \urlstyle{rm}\Url}\fi

\bibitem{akopyan2020-invariants}
Akopyan, A., Schwartz, R., Tabachnikov, S. (2020).
\newblock Billiards in ellipses revisited.
\newblock \emph{Eur. J. Math.}
\newblock {doi}:10.1007/s40879-020-00426-9.

\bibitem{akopyan2007-conics}
Akopyan, A.~V., Zaslavsky, A.~A. (2007).
\newblock \emph{Geometry of Conics}.
\newblock Providence, RI: Amer. Math. Soc.

\bibitem{armitage-2006}
Armitage, J.~V., Eberlein, W.~F. (2006).
\newblock \emph{Elliptic Functions}.
\newblock London: Cambridge University Press.

\bibitem{bialy2020-invariants}
Bialy, M., Tabachnikov, S. (2020).
\newblock {Dan Reznik's} identities and more.
\newblock \emph{Eur. J. Math.}
\newblock {doi}:10.1007/s40879-020-00428-7.

\bibitem{bos-1987}
Bos, H. J.~M., Kers, C., Raven, D.~W. (1987).
\newblock Poncelet's closure theorem.
\newblock \emph{Expo. Math.}, 5: 289--364.

\bibitem{caliz2020-area-product}
Chavez-Caliz, A. (2020).
\newblock More about areas and centers of {Poncelet} polygons.
\newblock \emph{Arnold Math J.}
\newblock Doi:10.1007/s40598-020-00154-8.

\bibitem{dragovic11}
Dragovi\'{c}, V., Radnovi\'{c}, M. (2011).
\newblock \emph{Poncelet Porisms and Beyond: Integrable Billiards,
  Hyperelliptic Jacobians and Pencils of Quadrics}.
\newblock Frontiers in Mathematics. Basel: Springer.

\bibitem{gallatly1914-geometry}
Gallatly, W. (1914).
\newblock \emph{The modern geometry of the triangle}.
\newblock London: Francis Hodgson.

\bibitem{garcia2020-self-intersected}
Garcia, R., Reznik, D. (2020).
\newblock Invariants of self-intersected and inversive {N}-periodics in the
  elliptic billiard.
\newblock arXiv:2011.06640.

\bibitem{garcia2020-new-properties}
Garcia, R., Reznik, D., Koiller, J. (2020).
\newblock New properties of triangular orbits in elliptic billiards.
\newblock \emph{Amer. Math. Monthly}, to appear.

\bibitem{stachel2019-conics}
Glaeser, G., Stachel, H., Odehnal, B. (2016).
\newblock \emph{The Universe of Conics: From the ancient Greeks to 21st century
  developments}.
\newblock Berlin: Springer.

\bibitem{johnson1960}
Johnson, R.~A. (1960).
\newblock \emph{Advanced Euclidean Geometry}.
\newblock New York, NY: Dover, 2nd ed.
\newblock Editor John W. Young.

\bibitem{nash2018-poncelet}
Nash, O. (2018).
\newblock Poring over {P}oncelet's porism.
\newblock \url{http://bit.ly/3r1rwxv}.

\bibitem{odehnal2011-poristic}
Odehnal, B. (2011).
\newblock Poristic loci of triangle centers.
\newblock \emph{J. Geom. Graph.}, 15(1): 45--67.

\bibitem{reznik2020-n3-focus-inversive}
Reznik, D., Garcia, R. (2020).
\newblock The talented {M}r. inversive triangle in the elliptic billiard.
\newblock \emph{Eur. J. of Math}.
\newblock To appear.

\bibitem{reznik2020-intelligencer}
Reznik, D., Garcia, R., Koiller, J. (2020).
\newblock Can the elliptic billiard still surprise us?
\newblock \emph{Math Intelligencer}, 42: 6--17.
\newblock {doi}:10.1007/s00283-019-09951-2.

\bibitem{reznik2021-fifty}
Reznik, D., Garcia, R., Koiller, J. (2021).
\newblock Fifty new invariants of {N}-periodics in the elliptic billiard.
\newblock \emph{Arnold Math. J.}
\newblock {doi}:10.1007/s40598-021-00174-y.

\bibitem{sergei91}
Tabachnikov, S. (2005).
\newblock \emph{Geometry and Billiards}, vol.~30 of \emph{Student Mathematical
  Library}.
\newblock Providence, RI: Am. Math. Society.

\bibitem{mw}
Weisstein, E. (2019).
\newblock Mathworld.
\newblock \emph{MathWorld--A Wolfram Web Resource}.
\newblock \url{mathworld.wolfram.com}.

\end{thebibliography}

\end{document}